\documentclass[
final,
]{amsart}

\usepackage[english]{babel}
\usepackage[utf8]{inputenc}

\usepackage[T1]{fontenc}
\usepackage{lmodern}
\usepackage{bm,upgreek}

\usepackage[inline]{enumitem}
\setlist[enumerate,1]{label={(\roman*)}}
\usepackage{verbatim}
\usepackage{url}
\usepackage{xcolor}
\usepackage[notref,notcite,color]{showkeys}
\usepackage{microtype}
\usepackage{varwidth} 
\usepackage{mathtools}

\usepackage[LogicalSystemsSans]{ak_logic}

\newcommand*{\Theorem}{Theorem}
\newcommand*{\Proposition}{Proposition}
\newcommand*{\Lemma}{Lemma}
\newcommand*{\Corollary}{Corollary}
\newcommand*{\Definition}{Definition}
\newcommand*{\Question}{Question}
\newcommand*{\Remark}{Remark}
\newcommand*{\Notation}{Notation}
\newcommand*{\Figure}{Figure}

\theoremstyle{plain}
\newtheorem{theorem}{\Theorem}
\newtheorem{proposition}[theorem]{\Proposition}
\newtheorem{corollary}[theorem]{\Corollary}
\newtheorem{lemma}[theorem]{\Lemma}
\theoremstyle{definition}

\theoremstyle{remark}
\newtheorem{remark}[theorem]{\Remark}

\usepackage{prettyref}
\newrefformat{thm}{\Theorem~\ref{#1}}
\newrefformat{pro}{\Proposition~\ref{#1}}
\newrefformat{cor}{\Corollary~\ref{#1}}
\newrefformat{lem}{\Lemma~\ref{#1}}
\newrefformat{rem}{\Remark~\ref{#1}}
\newrefformat{def}{\Definition~\ref{#1}}
\newrefformat{fig}{\Figure~\ref{#1}}

\newcommand{\T}{\mathbf{T}}
\newcommand{\meas}{\bm\upmu}
\newcommand{\leb}{\bm\uplambda}
\newcommand{\Cond}{\textup{Cond}}
\newcommand{\qfpre}{\textup{qf-}}
\newcommand{\symdiff}{\mathbin{\triangle}}

\title{Measure theory and higher order arithmetic}
\author{Alexander P.\ Kreuzer}
\address{Department of Mathematics \\
Faculty of Science \\
National University of Singapore \\
Block S17, 10 Lower Kent Ridge Road \\
Singapore 119076 
}
\email{matkaps@nus.edu.sg}
\urladdr{http://www.math.nus.edu.sg/~matkaps/}
\thanks{The author was partly supported by the RECRE project, and the Ministry of Education of Singapore through grant R146-000-184-112 (MOE2013-T2-1-062).}

\date{\today}

\begin{document}

\begin{abstract}
  We investigate the statement that the Lebesgue measure defined on all subsets of the Cantor-space exists.  
  As base system we take $\ls{ACA_0^\omega} +\lpf{\mu}$. The system $\ls{ACA_0^\omega}$ is the higher order extension of Friedman's system \ls{ACA_0}, and \lpf{\mu} denotes Feferman's $\mu$, that is a uniform functional for arithmetical comprehension defined by $f(\mu(f))=0$ if $\Exists{n} f(n)=0$ for $f\in \Nat^\Nat$.
  Feferman's $\mu$ will provide countable unions and intersections of sets of reals and is, in fact, equivalent to this. For this reason $\ls{ACA_0^\omega} + \lpf{\mu}$ is the weakest fragment of higher order arithmetic where $\sigma$-additive measures are directly definable.

  We obtain that over $\ls{ACA_0^\omega} + \lpf{\mu}$ the existence of the Lebesgue measure is $\Pi^1_2$-conservative over \ls{ACA_0^\omega} and with this conservative over \ls{PA}. Moreover, we establish a corresponding program extraction result.
\end{abstract}

\maketitle

In this paper, we will investigate the statement that the Lebesgue  measure, defined \emph{on all subsets} of the Cantor-space $2^\Nat$, or equivalently the unit interval $[0,1]$, exists. 
The setting of this investigation will be higher order arithmetic---as base system we will take $\ls{ACA_0^\omega} + \lpf{\mu}$. This is the higher order extension of Friedman's system for arithmetical comprehension \ls{ACA_0} together with Feferman's $\mu$, a functional which provides a uniform variant of arithmetical comprehension. 
The functional $\mu$ will be used to define countable unions and intersections of sets and is, in fact, equivalent to this.
With this, $\ls{ACA_0^\omega} + \lpf{\mu}$ is the weakest system in which the textbook definition of measures, including $\sigma$-additivity, can be formulated.
In addition to that it is strong enough to develop most of analysis since it contains \ls{ACA_0}, see \cite{sS09}.  Moreover, in the same setting also other investigations on higher order reverse mathematics have been carried out, see for instance \cite{jH08} and \cite{nSta}.
We therefore believe that $\ls{ACA_0^\omega} + \lpf{\mu}$ is a suitable system for investigating measure theory and the Lebesgue measure.

Our main result is that $\ls{ACA_0^\omega} + \lpf{\mu}$ plus the existence of the  Lebesgue  measure defined \emph{on all subsets} of the Cantor-space $2^\Nat$ (denoted by \lpf{\leb}) is $\Pi^1_2$\nobreakdash-conservative over \ls{ACA_0^\omega}. With this, $\ls{ACA_0^\omega} + \lpf{\mu} + \lpf{\leb}$ is conservative over Peano arithmetic (\ls{PA}). In addition to that our proof yields a method translating a derivation of a $\Pi^1_2$\nobreakdash-statement in $\ls{ACA_0^\omega} + \lpf{\mu} + \lpf{\leb}$ into a derivation in \ls{ACA_0^\omega}, and it exhibits program extraction (see \prettyref{thm:lebex} below).

This is interesting from a foundational viewpoint since it provides the consistency of the existence of the Lebesgue measure in a weak theory. For full Zermelo Fraenkel set theory without choice this is long known---the classical forcing construction of Solovay gives a model satisfying $\ls{ZF} + \lp{DC} + \lpf{\leb}$ (among other properties). The construction of this model, and in fact any model of \lp{ZF} in which all subsets of $2^\Nat$ are measurable, requires the existence of an inaccessible cardinal and goes, with this, even beyond \ls{ZFC}. See \cite{rS70} and \cite{sSh84,jR84}. As we will see this is not the case here.

Models of higher order arithmetic containing only measurable subsets of $2^\Nat$ can easily be constructed, see Theorem~7.6.5 of \cite{sF77}. However, adding the actual Lebesgue measure $\leb$ to the model/theory is difficult since it is a 4th order object and one has to trace the influence it has on 3rd and 2nd order objects. To the knowledge author this has not been considered, yet.
Usually one circumvents this problem by substituting measure theory with integration theory, see \cite{hF80} and \cite{YS90,BGS02}.

Also, our result has possible uses in applied proof theory and proof mining, i.e.\@ the analysis of proofs from mathematics using logical tools, see \cite{uK08} for an introduction. 
Here one encounters proofs using measure theory for instance in the analysis of ergodic theorems. 
In some cases the use of a measure can substitute by integration theory. Then the statement with measure is transformed to a statement about elements in suitable space, e.g.\@ $L_1$ or $L_2$. This is for instance the case in the analysis of the mean ergodic theorem, see \cite{AGT10,KL09}. 
In cases where this is not possible, for instance in the analysis of the pointwise ergodic theorem, see again \cite{AGT10}, the measure is explicitly used in the process of the extraction of a rate of meta-stability, without having a metatheorem which would guarantee a priori the success of this extraction.
Our result here provides such a metatheorem.

The result of this paper is obtained with a combination of the functional interpretation and program normalization. This technique was developed together with U.\@ Kohlenbach in \cite{KK12}. A similar approach was also used to eliminate ultrafilters in higher order arithmetic, see \cite{aK12b,aK15}.

The paper is organized as follows. In \prettyref{sec:log} the logical systems are introduced; in \prettyref{sec:meas} the existence of the Lebesgue measure is formalized in $\ls{ACA_0^\omega} + \lpf{\mu}$, the main result (Theorems \ref{thm:leb} and \ref{thm:lebex}) is stated and the proof is sketched; in Sections \ref{sec:con}, \ref{sec:pt} theorems and lemmas used in the proof are given; in \prettyref{sec:proof} the full proof of the result is given; in \prettyref{sec:baire} it is indicate how one can extend Theorems \ref{thm:leb} and \ref{thm:lebex} to include the statement that all sets have the property of Baire.

\section{Logical Systems}\label{sec:log}

We will work in fragments of Peano arithmetic in all finite types.
The set of all finite types $\T$ is defined to be the smallest set that satisfies 
\[
0\in \T, \qquad \rho,\tau\in \T \Rightarrow \tau(\rho)\in \T
.\]
The type $0$ denotes the type of natural numbers and the type $\tau(\rho)$ denotes the type of functions from $\rho$ to $\tau$. The pure types are denoted by natural numbers. They are given by
\[
0:= 0,\qquad n+1 := 0(n) .
\]
The degree of a type is defined by
\[
\textup{deg}(0):=0, \qquad \textup{deg}(\tau(\rho)) := \max(\textup{deg}(\tau),\textup{deg}(\rho)+1)
.\]
It is clear that $\textup{deg}(n) = n$ and that any object of type $\tau$ with $\textup{deg}(\tau)=n$ can be coded by an object of the pure type $n$. 
The type of a variable/term will sometimes be written as superscript.

Equality $=_0$ for type $0$ objects will be added as primitive notion to the systems together with the usual equality axioms.
Higher type equality $=_{\tau\rho}$ will be treated as abbreviation:
\[
x^{\tau\rho}=_{\tau\rho} y^{\tau\rho} :\equiv \Forall{z^\rho} xz =_\tau yz
.\]

Define the $\lambda$-combinators $\Pi_{\rho,\sigma}, \Sigma_{\rho,\sigma,\tau}$ for $\rho,\sigma,\tau\in\T$ to be the functionals satisfying 
\[
\Pi_{\rho,\sigma} x^\rho y^\sigma =_\rho x , \qquad \Sigma_{\rho,\sigma,\tau} x^{\tau\sigma\rho} y^{\sigma\rho}z^{\rho} =_\tau xz(yz)
.\]
Similar define the recursor  $R_\rho$ of type $\rho$ to be the functional satisfying 
\[
R_\rho 0yz =_\rho y, \qquad R_\rho (Sx^0)yz =_\rho z(R_\rho xyz)x
.\]
Let \emph{G\"odel's system $T$} be the $\T$-sorted set of closed terms that can be build up from $0^0$, the successor function $S^1$, the $\lambda$-combinators and, the recursors $R_\rho$ for all finite types $\rho$.
Using the $\lambda$-combinators one easily sees that $T$ is closed under $\lambda$-abstraction, see \cite{aT73}.
Denote by $T_0$ the subsystem of G\"odel's system $T$, where primitive recursion is restricted to recursors $R_0$.
The system $T_0$ corresponds to the extension of Kleene's primitive recursive functionals to mixed types, see \cite{sK59}, whereas full system $T$ corresponds to G\"odel's primitive recursive functionals, see \cite{kG58}.
By $T_0[F]$ we will denote the system resulting from adding a function(al) $F$ to $T_0$.

Let \lp[QF]{AC} be the schema
\[
\Forall{x^\rho}\Exists{y^\tau} \lf{A}_\qf(x,y) \IMPL \Exists{f^{\tau\rho}}\Forall{x^\rho} \lf{A}_\qf(x,f(x))
,\]
where $\lf{A}_\qf$ is a quantifier-free formula and $\rho,\tau\in \T$. 
If the types of $x,y$ are restricted to  $\rho,\tau$ we write \lp[QF]{AC^{\rho,\tau}}.
The schema \lp[QF]{AC^{0,0}} corresponds to recursive comprehension (\lp[\Delta^0_1]{CA}) in a second order context.

The system \ls{RCA_0^\omega} is defined to be the extension of the term system $T_0$ by $\Sigma^0_1$\nobreakdash-induction, \lp[QF]{AC^{1,0}}, and 
the extensionality axioms 
\[
(\lp{E_{\rho,\tau}})\colon\Forall{z^{\tau\rho} \!,x^\rho \!,y^\rho} (x=_\rho y \IMPL zx =_\tau zy)
\]
for all $\tau,\rho\in \T$.
(Strictly speaking the system \ls{RCA_0^\omega} was defined in \cite{uK05b} to contain only quantifier free induction instead of $\Sigma^0_1$\nobreakdash-induction. Since $\Sigma^0_1$\nobreakdash-induction is provable in that system, we may also add it directly.)
The system \ls{ACA_0^\omega} is defined to be $\ls{RCA_0^\omega}+\ls[\Pi^0_1]{CA}$.

These systems are conservative over their second order counterparts, where the second order part is given by functions instead of sets. These second order systems can then be interpreted in \ls{RCA_0}, resp.\@ \ls{ACA_0}. See \cite{uK05b}.

We will also make use of a uniform variant of \ls{ACA_0^\omega} given by Feferman's $\mu$. This is the functional of type 2 satisfying 
\begin{equation}\label{eq:mu}
  f(\mu(f))= 0\quad\text{if}\quad \Exists{x^0} f(x)=0 
.\end{equation}
We will denote the statement that Feferman's $\mu$ exists by \lpf{\mu}. For notational ease we will usually add a Skolem constant for $\mu$ and denote this statement also with \lpf {\mu}. It is straightforward to see that $\ls{RCA_0^\omega}+\lpf{\mu} \vdash \lp[\Pi^0_1]{CA}$. The reversal is not true, since \ls{ACA_0^\omega} does not contain enough choice to build such a functional. However, we have the following conservativity.

\begin{theorem}[{\cite[Theorem~2.5]{jH08}}]
  $\ls{ACA_0^\omega} + \lpf{\mu}$ is conservative over the second order theory \ls{ACA_0}.
\end{theorem}
\begin{remark}\label{rem:mucons}
  For $\Pi^1_2$ statements there exists a direct algorithm, based on the functional interpretation, rewriting proofs in $\ls{ACA_0^\omega} + \lpf{\mu}$ into proofs in \ls{ACA_0^\omega}, see \cite[Theorem~8.3.4]{AF98}.
\end{remark}

Real numbers will be presented as Cauchy reals coded by type $1$ functions. Equality and comparison will be written as $=_\Real$, $<_\Real$, $\le_\Real$. This relations are treated as abbreviations for the respective formulas. For details see \cite[Chapter~4]{uK08}.

\section{Measures in $\ls{RCA_0^\omega} + \lpf{\mu}$}\label{sec:meas}

We will code sets as characteristic functions. For instance a set $Y \subseteq \Nat$ is given by a function $\chi_Y$ of type $1$ such that
\[
\chi_Y(n) =_0 0 \quad\text{if{f}}\quad n\in Y \qquad \text{for all $n\in \Nat$},
\]
or a set $X\subseteq 2^\Nat$ is given by a function $\chi_X$ of type $2$ with
\[
\chi_X(\chi_Y) =_0 0 \quad\text{if{f}}\quad Y\in X.
\]
It is clear that each function $f^1$ is a characteristic function of a set $Y\subseteq \Nat$.
By replacing the characteristic function $\chi_Y$ with a normalization taking only values in $\{0,1\}$ each functional $F^2$ again is a characteristic function of a set $X\subseteq 2^\Nat$, i.e., $\lambda f^1. F(\lambda n^0 . \sg f(n))$ with $\sg(0) = 0$ and $\sg(m) = 1$ for $m>0$ is a characteristic function given by $F$. Thus, we can freely quantify over sets without adding any additional quantification.

The usual set theoretic operations like (finite) intersection, (finite) union and complement can then be defined using $\min()$, $\max()$, $1\dotminus x$. We will use the usual symbols $\cap$, $\cup$, $X^\complement$ as abbreviation for this. We will also use the abbreviation $X_1\setminus X_2 := X_1 \cap X_2^\complement$.

With the help of $\mu$ one also obtains countable intersections and unions. For instance, the characteristic function of $\bigcup_{n\in\Nat} X_n$ is given by
\[
\lambda f .\, \chi_{X_{N(f)}}(f) \quad\text{with}\quad N(f):= {\mu(\lambda n^0  \!.\, \chi_{X_n}(f))}
.\]

With this one can formulate in $\ls{RCA_0^\omega} + \lpf{\mu}$ the statement that a functional $\meas$ defines a measure on the Cantor-space $2^\Nat$. In more detail this is given by the following.
\begin{enumerate} 
\item non-negativity: $\Forall{X\in \mathcal{P}(2^\Nat)} \left(\meas(X) \ge_\Real 0\right)$,
\item $\meas(\emptyset)=_\Real 0$, and
\item $\sigma$\nobreakdash-additivity:
  \[
   \Forall{(X_n)_{n\in\Nat} \subseteq \mathcal{P}(2^\Nat)} \left(\Forall{i}\Forall{j>i} \left(X_i \cap X_j = \emptyset\right) \IMPL \meas\left(\bigcup_{n\in\Nat} X_n\right) =_\Real \sum_{n\in\Nat} \meas(X_n)\right)
   .\]
   Note that the limit in the sum is definable using $\mu$.
\end{enumerate}
A measure $\leb$ is the Lebesgue measure if it additionally satisfies the following.
\begin{enumerate}[resume]
\item Measure on basic open sets:
  \[
  \Forall{s\in 2^{<\Nat}} \left(\leb([s])=_\Real 2^{-\lth(s)}\right)
  \]
  where $[s]$ denotes the basic open set coded by $s$, i.e.,
  \[
  [s] := \left\{ X\subseteq 2^\Nat \sizeMid \Forall{i<\lth(s)} \left((s)_i = 0 \IFF i\in X\right) \right\} .
  \]
\end{enumerate}

We will denote by $\lpf{\leb}$ the statement the Lebesgue measure defined on all subsets of the Cantor-space exists. As we did with Feferman's $\mu$, we will usually add $\leb$ as Skolem constant to the system and denote this also with \lpf{\leb}. For later use we will write down the full statement of this axiom.
\[
\lpf{\leb}\colon \left\{ 
  \begin{aligned}
    \Exists{\leb^3} \Big( & \hphantom{\!\AND\,} \Forall{X^2} \left(\leb(X) \ge_\Real 0 
      \AND \left((\Forall{Y^1} X(Y)\neq _0 0) \IMPL \leb(X) =_{\Real} 0 \right)\right) \\
    &{\!\AND\,} \Forall{(X_i)_i^2} \left( \leb\left(\bigcup_{i\in\Nat} X'_i\right) =_\Real \sum_{i\in\Nat} \leb(X'_i)\right) \\
    &{\!\AND\,} \Forall{s^0} \leb([s])=_\Real 2^{-\lth(s)} \\
    &{\!\AND\,} \Forall{X^2} \leb(\lambda f^1\!.\, \sg(Xf)) =_\Real \leb(\lambda f^1\!.\, X(\lambda x^0\!.\, \sg(fx))) =_\Real \leb(X) \Big)
  \end{aligned}
  \right.
\]
where $X'_i := X_i \setminus \bigcup_{j<i} X'_j$. The last line indicates that $\leb$ is compatible with the coding of sets.

The Cantor-space and the unit interval are measure-theoretically equivalent. For instance the map $F\colon 2^\Nat \longto [0,1]$ with $F(g) := \sum_{i\in \Nat} 2^{-(i+1)} \cdot g(i)$ is a measure-preserving surjection. Therefore, \lpf{\leb} also gives the Lebesgue measure on the unit interval.

Since the axiom of choice proves the existence of non-measurable sets, $\lpf{\leb}$ is not consistent with \ls{ZFC}. However---as we already mentioned---by a classical result of Solovay there is, under the assumption that inaccessible cardinals exists, a model of $\ls{ZF}+ \lp{DC}$ in which each set is measurable and which, with this, satisfies $\lpf{\leb}$.

We are now in the position to state the main results of this paper.
\begin{theorem}\label{thm:leb}
  $\ls{ACA_0^\omega} + \lpf{\mu} + \lpf{\leb}$ is $\Pi^1_2$\nobreakdash-conservative over \ls{ACA_0^\omega}.
\end{theorem}
As immediate consequence we get that  $\ls{ACA_0^\omega} + \lpf{\mu} + \lpf{\leb}$ is consistent, even without the assumption of an inaccessible cardinal.

Since \ls{ACA_0^\omega} is conservative over \ls{ACA_0} which is conservative over \ls{PA}, see \cite[Remark~IX.1.7]{sS09}, \cite{hF76}, we immediately get the following corollary.
\begin{corollary}
  $\ls{ACA_0^\omega} + \lpf{\mu} + \lpf{\leb}$ is conservative over \ls{PA}.
\end{corollary}

We also have a program extraction result corresponding to \prettyref{thm:leb}.
\begin{theorem}[Program extraction for \lpf{\leb}]\label{thm:lebex}
  Suppose that
  \begin{equation}\label{eq:lebex}
  \ls{ACA_0^\omega} + \lpf{\mu} + \lpf{\leb} \vdash \Forall{f^1} \Exists{x^0} \lf{A}_\qf(f,x)
  \end{equation}
  for a quantifier-free formula $\lf{A}_\qf$ not containing $\leb$, then one can extract a term $t\in T_0[\mu]$ such that
  \[
  \Forall{f} \lf{A}_\qf(f,t(f))
  .\]
  If $\lf{A}_\qf$ additionally does not contain $\mu$ the term $t$ can be chosen to be in $T$.
\end{theorem}
\begin{remark}
  The proof will show that if one has \eqref{eq:lebex} for an $\lf{A}_\qf$ containing $\leb$ one can find a formula, equivalent over 
$\ls{ACA_0^\omega} + \lpf{\mu} + \lpf{\leb}$, which does not contain $\leb$. Thus \prettyref{thm:lebex} is in this sense also applicable in this case.
\end{remark}

We will now sketch the proof of \prettyref{thm:leb}. The full proof of Theorems \ref{thm:leb} and \ref{thm:lebex} will be given in \prettyref{sec:proof} after auxiliary results haven been stated and proven.

\begin{proof}[Sketch of proof of \prettyref{thm:leb}]
  Suppose $\ls{ACA_0^\omega} + \lpf{\mu} + \lpf{\leb}$ proves a $\Pi^1_2$\nobreakdash-statement $\Forall{f}\Exists{g} \lf{A}(f,g)$.
  \begin{enumerate}[label=\arabic*.]
  \item Using the functional interpretation in combination with the elimination of extensionality one obtains a term $t_g\in T_0[\mu,\leb]$ such that 
    \begin{equation}
      \label{eq:A}
      \lf{A}(f,t_g(f))
      .
    \end{equation}
    (See \prettyref{cor:extelim} and \prettyref{thm:ndint}.)
    The statement \eqref{eq:A} will be verifiable in a neutral (with respect to extensionality) and (so-called) quantifier-free system.
    See Sections~\ref{sec:elimex} and \ref{sec:nd} for details.
  \item The term $t_g$ will be normalized in such a way that each occurrence of $\leb$ is of the form
    \begin{equation}
      \label{eq:leb}
      \leb(t_0[\tilde{f}^1])
    \end{equation}
    where $\tilde{f}$ is the only free variable of $t_0$ and $\lambda \tilde{f}.\, t_0[f^1] \in T_0[\mu,\leb]$. (See \prettyref{thm:cut3}.)
  \item By the special properties of the neutral and quantifier-free system verifying \eqref{eq:A} we can find a derivation of this statement where each occurrence in the derivation of $\leb$ is of the form \eqref{eq:leb}. (See \prettyref{pro:cutderiv}.)
  \item The terms $t_0$ to which $\leb$ is applied to in \eqref{eq:leb} code arithmetical sets. Therefore the value of \eqref{eq:leb} can be calculated in $\ls{ACA_0^\omega} + \lpf{\mu}$. See \prettyref{lem:meas1}. By \prettyref{pro:cutderiv} the different occurrences of $\leb$ (as seen from the quantifier-free system) can be replaced independently with this calculation. This yields then a proof of $\lf{A}$ without using $\leb$.\qedhere
  \end{enumerate}
\end{proof} 

\section{Construction measures on $T_0$ and $T_0[\mu]$}\label{sec:con}

\begin{lemma}\label{lem:meas1}
  \begin{enumerate}
  \item\label{enum:meas1:1} Let $t[f^1] \in T$ such that for each $f$  it is coding a set $X\subseteq 2^\Nat$. A function satisfying the axioms for $\lambda f^1\! .\, \leb(t[f])$ is definable using a term $t'(f) \in T$.
  \item\label{enum:meas1:2} The statement of \ref{enum:meas1:1} is also true with $T$ replaced by $T_0$ and $T_0[\mu]$.
  \item\label{enum:meas1:3} Let $\upepsilon(X)$ be a function selecting an element of a set $X$ if $\leb(X) >_\Real 0$, i.e.
    \[
    \Forall{X\subseteq 2^\Nat} \left(\leb(X) >_\Real 0 \IMPL X(\upepsilon(X)) = 0\right)
    .\]
    For each term $t[f^1]$ in $T$ resp.~$T_0,T_0[\mu]$ there is a term $t'[f^1]$ in $T$ resp.~$T_0,T_0[\mu]$ which choses a value for $\upepsilon(t[f])$, consistent with the above axiom.
    In other words, if $t[f^1]$ codes a set of positive measure then $t'[f^1]\in t[f^1]$.
  \end{enumerate}
\end{lemma}
\begin{proof}
  \ref{enum:meas1:1}:
  The term $t[f]$ is provably in $T$ continuous, see \cite{bS71} and \cite[Proposition~3.57]{uK08}.
  Thus, there exists a modulus of continuity $\omega_{t[f]} \in T$.
  To $\omega_{t[f]}$ there exists a majorant $\omega^*_{t[f^M]}\in T$. 
  For details on majorants see \cite[Chapter~3]{uK08}. 
  (We use here that $f$ is only of type 1 and therefore has the generic majorant $f^M$.) Now $\omega^*_{t[f^M]}(\lambda x^0 \! .\, 1)$ is a modulus of uniform continuity. By the definition of $X$ this means that for two sets $Y_1,Y_2\subseteq \Nat$ which are equal on the interval $[0; \omega^*_{t[f^M]}(\lambda x^0 \! .\, 1)]$ we have $Y_1\in X\IFF Y_2 \in X$; 
  or in other words $Y_1\in X$ if{f} the basic open set  $[\overline{\chi_{Y_1}}\big(\omega^*_{t[f^M]}(\lambda x^0 \! .\, 1)\big)]$ is contained in $X$.
  Thus,
  \begin{equation}\label{eq:meas1def}
    \begin{split}
      \leb(X) &= 2^{-k} \cdot  \abs{\left\{ s\in 2^{k} \sizeMid [s]\subseteq X\right\}} \\
      &= 2^{-k} \cdot \abs{\left\{ s\in 2^{k} \sizeMid t[f](\lambda n .\, (s)_n) = 0 \right\}},
    \end{split}
  \end{equation}
  where $k:=\omega^*_{t[f^M]}(\lambda x^0 \!.\, 1)$. Now, it is straightforward to find a term describing $\leb(X)$.

  \noindent \ref{enum:meas1:2}:
  The statement for $T_0$ is clear. For $t[f]\in T_0[\mu]$, the function represented by $\lambda g^1 \!.\, t[f^1]g$ is arithmetic, see Theorem~8.3.4 and below in \cite{AF98}. In other words, the set $X$ is arithmetic and thus of the form
  \begin{equation}\label{eq:x}
    X= \bigcup_{n_1} \bigcap_{n_2} \bigcup_{n_3} \dots \bigcap_{n_i} X'_{n_1,\dots,n_i}
  \end{equation}
  for a sequence of sets $X'_{n_1,\dots,n_i}$ represented by a term in $T_0$. By \ref{enum:meas1:1} the value $\leb(X'_{n_1,\dots,n_i})$ is given by a term in $T_0$. Now,
  \[
  \leb(X) = \lim_{m_1\to\infty} \lim_{m_2\to\infty}  \cdots \lim_{m_i\to\infty} 
  \leb\big(\underbrace{\cup_{n_1<m_1} \cap_{n_2<m_2}  \dots \cap_{n_i<m_i} X'_{n_1,\dots,n_i}}_{X''_{m_1,m_2,\dots,m_i}}\big)
  \]
  Since $X''_{m_1,m_2,\dots,m_i}$ is the finite union and intersection of a set represented by a term in $T_0$ it is also represented in $T_0$. Therefore, by \ref{enum:meas1:1} the function $\lambda m_1,m_2,\dots,m_i .\, \leb(X''_{m_1,m_2,\dots,m_i})$ is given by a term in $T_0$. The limits can be defined using $\mu$ therefore the $\leb(X)$ is given by a term in $T_0[\mu]$.

  \noindent \ref{enum:meas1:3}: 
  First let $t$ be in $T$ or $T_0$. Then one can set in view of \eqref{eq:meas1def} the value $\upepsilon(X)$ to $\lambda x^0\! .\, (s)_x$ for an $s\in 2^k$ such that $[s]\subseteq X$ and to $\lambda x^0 \! .\, 0$ otherwise.

  Now let $t\in T_0[\mu]$. By the algorithm in the proof of Proposition 3.4 of \cite{ADR12} there exists a $\emptyset^{(k)}$-computable 0/1\nobreakdash-tree $T$ of measure $>\leb(X)/2$, such that the infinite branches $[T]$ are contained in $X$. Using $\mu$ one can easy find an infinite branch of $T$ which gives $\upepsilon(X)$.
\end{proof}
The points \ref{enum:meas1:1}, \ref{enum:meas1:2} of this lemma should be compared with the following well known results from descriptive set theory by  Tanaka \cite{hT67} and Sacks \cite{gS69}. For a represented space $\mathcal{X}$ and a $\Sigma^0_n$-set $P\subseteq \Nat^\Nat \times \mathcal{X}$ the following relation
\[
R(\alpha,r) :\equiv \leb(P_\alpha) > r \qquad \text{where }P_\alpha = \{ x\in X \mid (\alpha,x)\in P \}
\]
is also $\Sigma^0_n$. This lemma is just a translation of this result into our setting.

\section{Proof theory}\label{sec:pt}

\subsection{Elimination of extensionality}\label{sec:elimex}
To be able to use the functional interpretation and to replace $\leb$, we will need to eliminate the use of extensionality. For the functional interpretation it is sufficient to arrive at a weakly extensional system where the extensionality axioms are replaced by the so-called quantifier-free extensionality rule, see for instance \cite{uK08}. 
To replace $\leb$ however we will need a neutral system, where the extensionality axioms are deleted without any substitute. (Only the substitution of type $0$ objects remains extensional by the $=_0$-axioms.)

To stay consistent with the notation in \cite{uK08,KK12}, we will denote \ls{RCA_0^\omega} with $\EPAw+ \lp[QF]{AC^{1,0}}$, i.e., \EPAw is  $\ls{RCA_0^\omega}$ without quantifier-free choice. The $\lpfont{E\textup{-}}$ prefix indicates that the system is extensional.

The neutral variant \NPAw is defined like \EPAw with the exception that the defining equations of the functionals are given as substitution schemata
\begin{align*}
(\ls{SUB})&\colon \left\{
  \begin{aligned}
    t[\Pi xy] &=_0 t[x]  \\
    t[\Sigma xyz] &=_0 t[xz(yz)] \\
    t[R0yz] &=_0 t[y] \\
    t[R(Sx)yz] &=_0 t[z(Rxyz)x]
  \end{aligned}
\right. &\text{for all $t$ of type $0$,}
\intertext{and case-distinction functionals $(\Cond_\rho)_{\rho\in\T}$ for each type together with the substitution schemata}
(\lp{SUB_\Cond})&\colon \left\{
  \begin{aligned}
    t[\Cond_\rho(0^0 \!,x^\rho \!,y^\rho)] &=_0 t[ x] \\
    t[\Cond_\rho(Su,x^\rho \!,y^\rho)] &=_0 t[ y]
  \end{aligned}
  \right. 
  & \text{for all $t$ of type 0}
\end{align*}
are added. In the absence of (weak) extensionality these substitutions schemes are stronger than the defining equalities. Also, the $\Cond_\rho$ functional (for $\rho \neq 0$) cannot be defined without extensionality (or the recursor $R_\rho$ simulating it). Together with (weak) extensionality $\Cond_\rho$ can be defined from case-distinction on type $0$. Thus, \NPAw is a subsystem of \EPAw. See \cite[Section~2.3]{KK12}.

To define extensionality in the neutral system we define the extensional equality relation $=^e_\rho$ by induction on $\rho$:
\[
\left\{
  \begin{aligned}
    x=^e_0 y &:\equiv x=_0 y ,\\
    x=^e_{\tau\rho} y & :\equiv \Forall{u^\rho,v^\rho} \left( u=^e_\rho v \IMPL xu =^e_\tau xv \AND  xu =^e_\tau yv \right).
  \end{aligned}
\right.
\]
We say that an object $x^\rho$ is \emph{hereditarily extensional} if $x =^e_\rho x$. For details on this relation see \cite[Section~10.4]{uK08} and \cite{hL73} for a reference.
It is easy to see, that any object of degree $\le 1$ is extensional, see \cite[Lemma~10.41]{uK08}, and that each term $t[x_1,\dots,x_n]$ where $x_1,\dots,x_n$ are the only free variables is extensional if the variables are, i.e.,
\[
x_1 =^e x_1 \AND \dots \AND x_n=^e x_n \IMPL t[x_1,\dots,x_n] =^e  t[x_1,\dots,x_n]
,\]
see \cite[Lemma~10.42]{uK08}.

The \emph{elimination of extensionality translation} $\lf{A}_e$ of a formula $\lf{A}$ is given by relativizing all quantifiers to hereditarily extensional object, i.e.\@
\begin{enumerate*}[label=(\roman*)]
  \item $\lf{A}_e :\equiv \lf{A}$ for $\lf{A}$ prime,
  \item $(\lf{A} \mathrel{\Box}  \lf{B})_e :\equiv \lf{A}_e \mathrel{\Box}  \lf{B}_e$ for $\Box \in \{\land,\lor,\to\}$,
  \item $(\Exists{ x^\rho} \lf{A})_e :\equiv \Exists{x^\rho} ( x=^e_\rho x \AND \lf{A}_e)$, 
  \item $(\Forall{ x^\rho} \lf{A})_e :\equiv \Forall{x^\rho} ( x=^e_\rho x \IMPL \lf{A}_e)$.
\end{enumerate*}

\begin{proposition}[Elimination of extensionality translation, \cite{hL73}, \cite{uK08}]\label{pro:extelim}
  Let $\lf{A}$ be a sentence of $\mathcal{L}(\EPAw)$.
  If 
  \begin{align*}
    \EPAw + \lp[QF]{AC^{1,0}} &\vdash \lf{A}
    \shortintertext{then }
    \NPAw + \lp[QF]{AC^{1,0}} & \vdash \lf{A}_e
    .
  \end{align*}
\end{proposition}
\begin{proof}
  See Proposition~10.45, Remark~11.46.1 and the whole Section~11.5 of \cite{uK08}.
\end{proof}
We may assume that $\mu$ satisfies 
\[
\mu(f) =
\begin{cases}
  \min \{ x \mid f(x) = 0 \} & \text{if such an $x$ exists,} \\
  0 & \text{otherwise,}
\end{cases}
\]
since such a functional is definable using $\mu$ satisfying only \eqref{eq:mu}. 
Then the values of $\mu(f)$ only depends on the values of $f$ and thus $\mu$ is hereditarily extensional. 
This means that one can add \lpf{\mu} to the systems in \prettyref{pro:extelim}.

For \lpf{\leb} we first observe that \lpf{\leb} proves over \NPAw that $\leb$ is extensional, since for two sets $X_1,X_2$ with $\Forall{f^1} X_1(f) =_0 X_2(f)$ we have that the characteristic function of $X_1\setminus X_2$ and $X_2\setminus X_1$ are never equal to $0$, thus by the first line of \lpf{\leb}
\[
\leb(X_1\setminus X_2) =_\Real \leb(X_2\setminus X_1) =_\Real 0
.\]
Now by the second line of the axiom \lpf{\leb} we have that
\[
\leb(X_1) + \leb(X_2\setminus X_1) =_{\Real} \leb(X_2) + \leb(X_1\setminus X_2)
\]
which means that $\leb(X_1) =_{\Real} \leb(X_2)$.
Thus we get together with the fact that degree $\le 1$ object are extensional that
\[
\lpf{\leb}_e\IFF \left\{ 
  \begin{aligned}
    \Exists{\leb^3} \Big( & \hphantom{\!\!\!\AND\,} \Forall{X^2}\!\left(X =^e X \IMPL \left(\leb(X) \ge_\Real 0 \AND ((\Forall{Y^1}\! X(Y) \neq_0 0) \IMPL \leb(X) =_{\Real} 0) \right)\right) \\
    &{\!\!\!\AND\,} \Forall{(X_i)_i^2} \left((X_i)_i =^e (X_i)_i \IMPL  \leb\left(\bigcup_{i\in\Nat} X'_i\right) =_\Real \sum_{i\in\Nat} \leb(X'_i)\right) \\
    &{\!\!\!\AND\,} \Forall{s^0}\, \leb([s])=_\Real 2^{-\lth(s)} \\
    &{\!\!\!\AND\,} \Forall{X^2} 
    \begin{multlined}[t]
      (X =^e X \\
      \IMPL \leb(\lambda f^1 \!.\, \sg(Xf)) =_\Real \leb(\lambda f^1 \!.\, X(\lambda x^0 \!.\, \sg(fx))) =_\Real \leb(X)) \smash{\Big)}
    \end{multlined}
  \end{aligned}
  \right.
\]
Since that relativisation is only applied to $\forall$-quantifiers, we get that $\lpf{\leb}\IMPL \lpf{\leb}_e$.

Combining these observations with \prettyref{pro:extelim} we obtain the following corollary.
\begin{corollary}\label{cor:extelim}
  Let $\lf{A}$ be a sentence of $\mathcal{L}(\EPAw + \lpf{\mu} + \lpf{\leb})$. If
  \begin{align*}
    \EPAw + \lp[QF]{AC^{1,0}} + \lpf{\mu} + \lpf{\leb} &\vdash \lf{A}
    \shortintertext{then }
    \NPAw + \lp[QF]{AC^{1,0}} + \lpf{\mu} + \lpf{\leb} & \vdash \lf{A}_e
    .
  \end{align*}
\end{corollary}

\subsection{Functional interpretation}\label{sec:nd}

In order to be able to later replace the occurrences of $\leb$ independently we need a stronger formulation of the functional interpretation, which states explicitly that it translates the proof into the quantifier-free fragment.

The quantifier-free fragment $\qfpre\NPAw$ of \NPAw is obtained as follows:
\begin{itemize}
\item The quantifier-rules and -axioms are dropped from logic.
\item For all axioms of the form $\lf{A}_\qf(x_1^{\rho_1},\dots,x_n^{\rho_n})$, where $\lf{A}_\qf$ is quantifier-free, the following axioms are added to the system:
  \[
  \lf{A}_\qf(t_1^{\rho_1},\dots,t_n^{\rho_n})
  ,\]
  where $t_i$ are arbitrary terms.
  
\item The induction schema is replaced by the (quantifier-free) induction rule:
  \[ 
  \frac{\lf{A}_\qf(0^0),\quad \lf{A}_\qf(x^0)\IMPL \lf{A}_\qf(Sx^0)}{\lf{A}_\qf(t^0)}
  ,\]
  where $\lf{A}_\qf$ is quantifier-free, $x$ does not occur free in the assumption and $t$ is an arbitrary term.
\end{itemize}
The system $\qfpre\NPAw$ contains only prime formulas of the form
\[
t_0 =_0 t_1
,\]
where $t_0,t_1$ are terms in \NPAw. Formulas are logical combinations of these predicates.
Obviously,  $\qfpre\NPAw$ is a subsystem of \NPAw.
The important property that our proof will rely on is that in $\qfpre\NPAw$ variables of type $>0$ cannot be instantiated by the axioms.

Therefore, one can assume that a derivation of a formula $\lf{A}$ contains besides type $0$ variables only the variables already occurring in $\lf{A}$, see \cite[Lemma~4]{KK12}.
For a detailed discussion of on quantifier-free systems we also refer the reader to \cite[1.6.5]{aT73} (We use here the variant of the systems described Remark~1.5.8.).

\begin{theorem}[Functional interpretation, \cite{kG58}, {\cite[8.5.2, 8.6.2]{sF77}}]\label{thm:ndint}
  If
  \begin{align*}
    \NPAw + \lp[QF]{AC} &\vdash \Forall{x} \Exists{y} \lf{A}_\qf(x,y),
    \shortintertext{then one can extract a closed term $t$ from the derivation, such that}
    \qfpre\NPAw & \vdash \lf{A}_\qf(x,tx).
  \end{align*}

  The same statement holds if to both system \lpf{\mu} is added.
\end{theorem}
This theorem is obtain by using a negative translation and G\"odel's Dialectica translation.
For a proof see Corollary~10.54 and the whole Chapter~10 of \cite{uK08} and note that one can simply drop weak extensionality. For the interpretation of $\mu$ see also \cite[8.2.2]{AF98}.

\subsection{Term normalization}\label{sec:cut}

In this section we formulate the term normalization results we will be using. \prettyref{thm:cut3} below shows that one can normalize each term containing a type $3$ parameter $G$---think of this being $\leb$---such that $G$ occurs only in the form $G(\tilde{t}_0[f^1])$. For $\leb$ this means that $\leb$ is only used to introduce functions, which can be built according to \prettyref{lem:meas1}.
\prettyref{pro:cutderiv} below shows that one can do this normalization in a whole derivation in $\qfpre\NPAw$ and is allowed to independently substitute different occurrences of $\leb$.

\begin{theorem}[term-normalization for degree $3$]\label{thm:cut3}
  Let $G$ be a constant of type~$3$.
  For every term $t^1\in T_0[(\Cond_\rho)_{\rho\in \T},G]$ there is provably in $\qfpre\NPAw$ a term $\tilde{t}\in T_0[\Cond_0,G]$ for which
  \[
  \Forall{x}\, tx =_0 \tilde{t}x
  \]
  and where
  every occurrence of a $G$ is of the form
  \[
  G(\tilde{t}_0[f^1])
  .\]
  where $\tilde{t}_0[f^1]$ contains only $f^1$ free.
\end{theorem}
\begin{proof}
  See \cite[Theorem~24]{KK12}, see also \cite[Proof of Proposition~4.2]{uK99}.
\end{proof}

\begin{proposition}\label{pro:cutderiv}
  Let $G$ be a new type $3$ constant and let $\lf{A}$ be a formula 
  with \[
  \qfpre\NPAw \vdash \lf{A}
  \]
  and containing only type~$0$ variables free.

  There exists a formula $\tilde{\lf{A}}$ such that $\EPAw$ proves $\lf{A} \IFF \tilde{\lf{A}}$ and that there is a derivation $\mathcal{\tilde{D}}$ of $\qfpre\NPAw \vdash \tilde{\lf{A}}$ where every occurring term is normalized according to \prettyref{thm:cut3}, i.e.\@ each occurrence of $G$ is of the form $G(t_i[f^1])$ for a term $t_i$.

  Moreover, these applications of $G$ can be chosen independently from each other in the sense that
  \[
  \qfpre\NPAw \nvdash P[G(t')] =_0 P[G(t'')] \qquad\text{for a fresh variable $P$}
  \]
  for all type 0 substitution instances $t', t''$ of $t_i$ resp.\ $t_j$ with $i\neq j$. (In other words, \qfpre\NPAw  does not see that the $G(t'),G(t'')$ are applications of $G$ and not just an arbitrary term of suitable type and with the same free variables. Hence they may be replaced independently.)

\end{proposition}
\begin{proof}
  See Proposition~26 and Remark~27 in \cite{KK12}.
\end{proof}

\section{The Proofs}\label{sec:proof}
\begin{proof}[Proof of \prettyref{thm:leb}]
  Suppose 
  \[
  \ls{ACA_0^\omega} + \lpf{\mu} + \lpf{\leb} \vdash \Forall{f^1} \Exists{g^1} \lf{A}_\textit{ar}(f,g)
  \]
  where $\lf{A}_\textit{ar}$ is arithmetic and contains only $f,g$ free. Using $\mu$ one can find a quantifier-free formula $\lf{A}_\qf(f,g)$ equivalent to $\lf{A}_\textit{ar}$. Thus,
  \[
  \ls{RCA_0^\omega} + \lpf{\mu} + \lpf{\leb} \vdash \Forall{f^1} \Exists{g^1} \lf{A}_\qf(f,g)
  .\]
  By \prettyref{cor:extelim} we get
  \begin{equation}\label{eq:1}
    \NPAw + \lp[QF]{AC^{1,0}} + \lpf{\mu} + \lpf{\leb} \vdash \Forall{f^1} \Exists{g^1} \lf{A}_\qf(f,g)
  .\end{equation}
  Now taking a closer look at the axiom \lpf{\leb} we note that real number equality and other quantification over type $0$ variables can be regarded as quantifier-free relative to \lpf{\mu}. Thus, with \lp[QF]{AC} we may build a choice function $\upepsilon$ for $Y$ and arrive at the following equivalent formula.
  \[
  \left\{ 
  \begin{aligned}
    \Exists{\leb^3 \!, \upepsilon^3} \Big( & \hphantom{\!\AND\,} \Forall{X^2} \left(\leb(X) \ge_\Real 0 
      \AND \left(X(\upepsilon(X))\neq_0 0) \IMPL \leb(X) =_{\Real} 0 \right)\right) \\
    &{\!\AND\,} \Forall{(X_i)^2} \left( \leb\left(\bigcup_{i\in\Nat} X'_i\right) =_\Real \sum_{i\in\Nat} \leb(X'_i)\right) \\
    &{\!\AND\,} \Forall{s^0} \leb([s])=_\Real 2^{-\lth(s)} \\
    &{\!\AND\,} \Forall{X^2} \leb(\lambda f^1 \!.\, \sg(Xf)) =_\Real \leb(\lambda f^1 \!.\, X(\lambda x^0 \!.\, \sg(fx))) =_\Real \leb(X) \Big)
  \end{aligned}
  \right.
  \]
  Coding $X,(X_i)$ into one type $2$ object $X$ and using $\mu$ to get rid of quantification over type $0$ variables we can rewrite the formula into the following form $\lpf{\leb}_\qf[\leb^3, \upepsilon^3, X^2]$ where $\lpf{\leb}_\qf$ is quantifier-free. Thus, we get 
  \[
  \NPAw + \lp[QF]{AC} + \lpf{\mu} \vdash \lpf{\leb} \IFF \Exists{\leb^3 \!, \upepsilon^3} \Forall{X^2} \lpf{\leb}_\qf(\leb, \upepsilon, X) 
  .\]
  
  Using the deduction theorem in \eqref{eq:1} and applying this we arrive at
  \[
  \NPAw + \lp[QF]{AC} + \lpf{\mu} \vdash \left(\Exists{\leb^3 \!, \upepsilon^3} \Forall{X^2} \lpf{\leb}_\qf(\leb, \upepsilon, X)\right) \IMPL \Forall{f^1} \Exists{g^1} \lf{A}_\qf(f,g) 
  .\]
  Pulling out the quantifiers yields
  \[
  \NPAw + \lp[QF]{AC} + \lpf{\mu} \vdash \Forall{\leb^3 \!, \upepsilon^3 \!,f^1} \Exists{X^2 \!,g^1} \big(\lpf{\leb}_\qf(\leb, \upepsilon, X) \IMPL \lf{A}_\qf(f,g) \big)
  .\]
  The functional interpretation, \prettyref{thm:ndint}, gives now terms $t_X[\leb^3 \!, \upepsilon^3 \!,f^1], t_g[\leb^3 \!, \upepsilon^3 \!,f^1]$ and
  \[
  \qfpre\NPAw + \lpf{\mu} \vdash \lpf{\leb}_\qf(\leb, \upepsilon, t_X[\leb, \upepsilon,f]) \IMPL \lf{A}_\qf(f,t_g[\leb, \upepsilon,f]) 
  .\]
  Now by \prettyref{pro:cutderiv} there exists terms $(t_i)$, and a derivation of a---over \EPAw equivalent---formula, such that each application of $\leb, \upepsilon$ is of the form $\leb(t_i[\tilde{f}^1]),\upepsilon(t_i[\tilde{f}^1])$ for a fresh $\tilde{f}^1$. By \prettyref{pro:cutderiv} these terms can be replaced independently.  Thus, we replace them, starting with the innermost, with the construction from \prettyref{lem:meas1}. With this the substituted instances of the premise become provable and we get
  \[
  \EPAw +  \lpf{\mu} \vdash \lf{A}_\qf(f,t'_g[f])
  \]
  where $t_g'$ is the normalization of $t_g$ where the applications of $\leb, \upepsilon$ have been replaced.
\end{proof}

\begin{proof}[Proof of \prettyref{thm:lebex}]
  For a proof of \prettyref{thm:lebex} note that the proof of \prettyref{thm:leb} actually extracts the program $\lambda f.\, t'_g[f]\in T_0[\mu]$. If $\lf{A}_\qf$ does not contain $\mu$ then one can prove the totality of $t'_g[f]$ in \ls{ACA_0^\omega}, see \prettyref{rem:mucons}. Thus we effectively get a derivation for 
  \[
  \ls{ACA_0^\omega} \vdash \Forall{f} \Exists{g} \lf{A}_\qf(f,g)
  .\]
  Now one can apply the functional interpretation and extract a term $t\in T_0[B_{0,1}]$ where $B_{0,1}$ is the bar-recursor of lowest type, see Section~11.3 of \cite{uK08}. Since $t$ is only of degree $2$ one can find an equivalent term $t'\in T$, see \cite[Corollary~4.4.(1)]{uK99}.
\end{proof}

\begin{remark}
  Looking at the proof of the Theorems \ref{thm:leb} and \ref{thm:lebex} one immediately sees that one can add any degree $2$ functional $F$ to the theories if the respective variant of \prettyref{lem:meas1} stays intact, i.e.~for each sequence of sets indexed by $f^1$ coded by a term $t[f]\in T_0[\mu, F]$ the functions $\lambda f .\, \leb(t[f])$ and $\lambda f.\, \upepsilon(t[f])$ are definable using a term $t'[f]\in T_0[\mu, F]$.

  From this it is obvious that one can add the transfinitely iterated Turing jump operator 
  \[
  \textup{TJ}(\alpha^0,X^1):= X^{(\alpha)},
  \]
  where $\alpha$ is the code of an ordinal, to the theories in the Theorems \ref{thm:leb} and \ref{thm:lebex}.
  Adding this operator allows one to define all sets definable in \lp{ATR_0}. However one cannot iterate the use of $\leb$ transfinitely, since it is, a priori, not coded by the Turing jump.
  Also, one can take for $F$ any other provably measurable function.
\end{remark}

\section{Extension: Baire property}\label{sec:baire}
In this section we will describe how to extend by the statement that all sets of reals have the property of Baire (\lp{BP}).
The same techniques as for \lpf{\leb} apply since the collection of all sets with the property of Baire forms also a $\sigma$\nobreakdash-algebra.

Recall that a set $X$ has the property of Baire if there exists an open set $G$ such that the symmetric difference $X \symdiff G$ is meager, i.e.\@ is the countable union of nowhere dense sets. 
Since a set is nowhere dense if{f} its closure is, a set is meager if{f} it is contained in a countable union of \emph{closed}, nowhere dense sets.

Now we can code an open set $G$ by a type $1$ function $o_G$ listing basic open sets building $G$, or in other words $G:= \bigcup_{n\in\Nat} [o_G(n)]$. With this, membership $f\in G$ is given by the $\Sigma^0_1$-formula $\Exists{n^0} o_G(n) = \langle f(0),\dots,f(\lth(o_G(n)))\rangle$. In a similar way, one can check whether a basic open set $[s]$ intersects $G$. Closed sets can then be coded as complements of open sets.
With this, we can formulate \lp{BP} in the following way.
\[
\lp{BP}\colon \left\{
  \begin{multlined}
    \Forall{X^2\subseteq 2^{\Nat}} \Exists{G^1\,\text{open}} \Exists{(H_n)_n^1\,\text{closed}}  \\
    \left(\Forall{n^0,s^0} \left([s] \nsubseteq H_n \right) \AND \Forall{f^1} \left(f\in X \symdiff G \IMPL \Exists{n} f\in H_n\right) \right) .
  \end{multlined}
\right.
\]

Since quantification over $X$ is the only quantification over degree $>1$, we get
\[
\lp{BP}_{\! e} \IFF \left\{
  \begin{multlined}
    \Forall{X^2\subseteq 2^{\Nat}} \big(X =^e X \IMPL \Exists{G^1\text{ open}} \Exists{(H_n)_n^1\text{ closed}}  \\
    \left(\Forall{n^0,s^0} \left([s] \nsubseteq H_n \right) \AND \Forall{f^1} \left(f\in X \symdiff G \IMPL \Exists{n} f\in H_n\right) \right) \big).
  \end{multlined}
\right.
\]
Thus, we have that $\lp{BP} \IMPL \lp{BP}_{\! e}$ and we can add \lp{BP} to the systems in \prettyref{cor:extelim}.

We will now show how to extend \prettyref{thm:leb} to also include \lp{BP}. 
For this suppose that 
\[
\ls{ACA_0^\omega} + \lpf{\mu} + \lpf{\leb} + \lp{BP} \vdash \Forall{f^1} \Exists{g^1} \lf{A}_\qf(f,g)
\]
for a $\Pi^1_2$ statement $\Forall{f^1} \Exists{g^1} \lf{A}_\qf(f,g)$.

By the previous considerations we have that already $\NPAw + \lpf{\mu} + \lpf{\leb} + \lpf{BP}$ proves $\Forall{f^1} \Exists{g^1} \lf{A}_\qf(f,g)$.
Now we could strengthen \lp{BP} to the following uniform variant.
\[ 
\lp{BP_{\!\mathit{u}}} \colon
\left\{
  \begin{multlined}
    \Exists{\mathcal{G}, \mathcal{H}_n} \Forall{X^2\subseteq 2^{\Nat}}  \\
    \left(\Forall{n^0,s^0} \left([s] \nsubseteq \mathcal{H}_n(X) \right) \AND \Forall{f^1} \left(f\in X \symdiff \mathcal{G}(X) \IMPL \Exists{n} f\in \mathcal{H}_n(X)\right) \right) .
  \end{multlined}
\right.
\]
(Note that we make use here of the fact that $\mathcal{G}, \mathcal{H}_n$ may be non-extensional.)
Again, \lp{BP_{\!\mathit{u}}} can be written, modulo $\mu$ and coding, in the form $\Exists{\mathcal{G}, \mathcal{H}_n} \Forall{X^2} \lp{BP_{\!\mathit{u,qf}}}(\mathcal{G}, \mathcal{H}_n,X)$. 

Arguing as in the proof of \prettyref{thm:leb} we obtain using the functional interpretation terms $t_X[\leb, \upepsilon,\mathcal{G}, \mathcal{H}_n, f],   t_X'[\leb, \upepsilon,\mathcal{G}, \mathcal{H}_n, f],  t_g[\leb, \upepsilon,\mathcal{G}, \mathcal{H}_n, f]$, such that
\begin{multline*}
  \qfpre\NPAw + \lpf{\mu} \vdash  \\ \lpf{\leb}_\qf(\leb, \upepsilon, t_X[\leb, \upepsilon,\mathcal{G}, \mathcal{H}_n,f]) \AND \lp{BP_{\!\mathit{u,qf}}}(\mathcal{G}, \mathcal{H}_n, t_X'[\leb, \upepsilon,\mathcal{G}, \mathcal{H}_n, f]) \\  \IMPL \lf{A}_\qf(f,t_g[\leb, \upepsilon,f]) 
.
\end{multline*}
The terms can be normalized such that each application of $\leb,\epsilon,\mathcal{G}, \mathcal{H}_n$ is of the form $\mathcal{G}(t[f^1])$ respectively,  for a fresh $f$. Again these term will be replaced starting with the innermost. Occurrences of $\mathcal{G}, \mathcal{H}_n$ will be replaced by the following.

Suppose $t\in T_0$. Then the set $X$ coded by $t$ is open thus setting $\mathcal{G}(t) = X$ and $\mathcal{H}_n(t)=\emptyset$ suffices.
If $t\in T_0[\mu]$ then the set $X$ is of the form \eqref{eq:x}. To build $\mathcal{G}(t), \mathcal{H}_n(t)$ is suffices to note the following.
\begin{itemize}
\item Countable unions can be handled by
  \[
  \mathcal{G}\Big(\bigcup\nolimits_{n}\! X_n\Big) := \bigcup\nolimits_n \!\mathcal{G}(X_n), \qquad \mathcal{H}_{\langle n_1,n_2\rangle}\Big(\bigcup\nolimits_{n}\! X_n\Big) := \mathcal{H}_{n_1}(X_{n_2}),
  \]
  since the union of open sets is open and a countable collection of countably many nowhere dense sets is countable.
\item Inversion can be handled by 
  \begin{align*}
    \mathcal{G}(\mathcal{2^\Nat} \setminus X) &:= \text{``Interior of $\mathcal{2^\Nat} \setminus \mathcal{G}(X)$''}, \\
    \mathcal{H}_n(\mathcal{2^\Nat} \setminus X) &:=
    \begin{cases}
      \text{``Border of $\mathcal{2^\Nat} \setminus \mathcal{G}(X)$''} & \text{if $n=0$,} \\
      \mathcal{H}_{n-1}(X) & \text{if $n>0$.}
    \end{cases}
  \end{align*}
  It is easy to see that this is definable using a term in $T_0[\mu]$. It is correct, since the topological border is nowhere dense and $(\mathcal{2^\Nat} \setminus X) \symdiff (\mathcal{2^\Nat} \setminus \mathcal{G}(X)) = X \symdiff \mathcal{G}(X)$ and thus $(\mathcal{2^\Nat} \setminus X) \symdiff \text{``Interior of $\mathcal{2^\Nat} \setminus \mathcal{G}(X)$''} \subseteq  \bigcup_n\mathcal{H}_n(\mathcal{2^\Nat} \setminus X)$.
\end{itemize}

In total this gives the following theorem.
\begin{theorem}
  Theorems \ref{thm:leb} and \ref{thm:lebex} remain true if \lp{BP} is added.
\end{theorem}

\bibliographystyle{amsalpha}
\bibliography{../bib}

\end{document}